\documentclass[10pt]{article}
\usepackage{amssymb,amsmath,amsopn,amsthm, xcolor,graphicx,tikz-cd}
\usepackage[mathscr]{euscript}

\usepackage{algorithm}
\usepackage{algorithmic}

\newtheorem{Theorem}{Theorem}

\newtheorem{Cor}{Corollary}
\theoremstyle{remark}
\newtheorem{rmk}{Remark}

\theoremstyle{definition}
\newtheorem{defn}{Definition}

\newtheorem{ex}{Example}

\title{Structure on the Top Homology and Related Algorithms}
\date{}
\author{Nissim Ranade, Chandrika Sadanand and Dennis Sullivan}

\begin{document}

\maketitle

\begin{abstract}

\noindent We explore the special structure of the top-dimensional homology of any compact triangulable space $X$ of dimension $d$. Since there are no $(d+1)$-dimensional cells, the top homology equals the top cycles and  is thus  a free abelian group. There is no obvious basis, but we show that there is a canonical embedding of the top homology into a canonical free abelian group which has a natural basis up to signs. This embedding structure is an invariant of $X$ up to homeomorphism. This circumstance gives the top homology the structure of an (orientable) matroid, where cycles  in the sense of matroids correspond to the cycles  in the sense of homology. This adds a novel topological invariant to the topological literature.

\medskip\noindent We apply this matroid structure on the top homology to give a polynomial-time algorithm for the construction of a basis of the top homology (over $\mathbb{Z}$ coefficients).

\end{abstract}

\section{Introduction}

\medskip\noindent This paper provides, for a triangulable  space, a cubic (in the number of cells) time algorithm which finds a basis of cycles in the top dimension, with $\mathbb{Z}$ coefficients. Since each top homology class is represented by a unique cycle, this basis of cycles is also a basis for the top homology. The algorithm takes advantage of an oriented matroid structure on the set of components of the oriented top-dimensional  pseudo-manifold subsets, which we call \emph{strata}. The elements of the basis produced are minimally supported on the strata. This follows from a topological picture of $X$, showing the entire structure is a homeomorphism invariant of $X$.

\medskip\noindent We  begin with a motivating  example.

\begin{ex}
\label{graph}
Consider the problem of finding cycles in a finite graph. A straight-forward approach would be to look at all subsets of edges and see which subsets are cycles. This however, could be very slow, since the time it takes is exponential in the number of edges. The following is a more efficient way of finding cycles in the graph $G$.

\medskip\noindent Start with two empty sets $M$ and $S$. Elements will be added to these sets so that $M$ is an acyclic set of edges (a tree) and $S$ is a set of cycles.

\medskip\noindent Consider the vector space $V$ generated by the vertices of $G$. Note that $M$ does not contain a cycle if and only if $\{ \partial e \mid e \in M\} \subset V$ is linearly independent. If a set of boundary vectors $\{\partial e_1, \ldots \partial e_n \}$ is linearly dependent with dependence relation $\sum a_i \partial e_i=0$, then $\sum a_i e_i$ forms the corresponding cycle in the graph (edges counted with multiplicity). Whitney generalized this observation  when he defined the notion of a \emph{matroid} \cite{whitney}.

\medskip\noindent Now consider the edges of $G$, one at a time.
  \begin{itemize}
  \item For an edge $e$ if $\{e\} \cup M$ has a cycle, then $\{\partial e\} \cup \{\partial f \mid f \in M\} \subset V$ has a minimal linear dependence relation. Add the corresponding cycle of edges to $S$.
 \item For an edge $e$, if $\{e\} \cup M$ does not contain a cycle, we add $e$ to $M$.
  \end{itemize}

\medskip\noindent Once the above process is repeated for all the edges of the graph, we have a maximal tree $M$ and a basis of cycles $S$.

\begin{figure}[h]
\begin{center}
\includegraphics[width=4in]{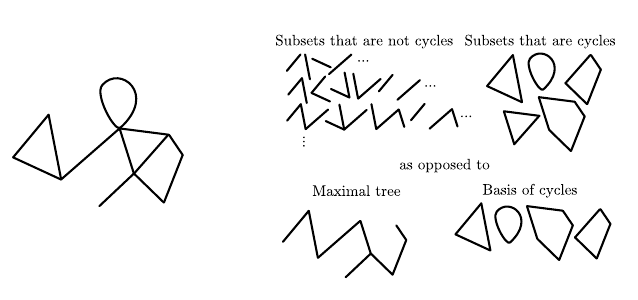}
\end{center}
\end{figure}
\end{ex}

\medskip\noindent This algorithm is $O(n^3)$ in the number of edges instead of exponential. Every cycle in the graph is a sum of directed cycles found by the algorithm. 

\medskip\noindent We generalize this method for finding cycles in graphs to finding top-dimensional cycles in cell complexes, using a matroid. In an effort to save time, in the example above, it is natural to amalgamate  pairs of edges that share a two-valent vertex. In our higher-dimensional context, the analogue of this amalgamation is to form the \textit{strata} mentioned earlier. Note that if the graph is weighted, then ordering the cells by increasing weight results in a minimum weight spanning tree. Our algorithm finds a higher-dimensional analogue of a minimum weight spanning tree.

\medskip\noindent Matroids were first defined by H. Whitney in 1935 \cite{whitney} to generalize the notion of linear dependence and cycles in graphs. Oriented matroids were introduced later around 1975 \cite{blandlasvergnas}, \cite{folkmanlawrence}. Both have since been used to advance the field of linear programming by showing that certain optimization problems can be solved using the greedy algorithm \cite{matroidtheory}. The matroid structure on the set of homology classes has been used extensively \cite{Greedy-optimal-homotopy-2005,Efficient-algorithms-2018,ChenFreedman}. The third author learned from Ralph Reid at MIT in the early 70's, that the top homology had a further geometric structure expressible in the language of matroids. In this paper, we prove the existence of a matroid structure on a more geometric set: the set of strata. We give an algorithm that uses this structure to compute the top dimensional homology in cubic time.


\medskip\noindent Computing the homology of a finite cell complex reduces to computing the Smith normal form of the boundary matrix \cite[Section 1.11]{Munkres}. However, this computation cannot be completed in a reasonable amount of time when there is a large number of cells, so several algorithms have been developed to reduce the number of cells in a complex while preserving its homology. This can be done by producing a large acyclic subcomplex (like the tree in Example \ref{graph}) and looking at the relative homology, and by amalgamating pairs of cells that form the coboundary of a third cell \cite{FIdF}. Our  amalgamation of top dimensional cells into strata is maximal in a sense. Any further amalgamation (amalgamating two strata) would not, in general, preserve the top dimensional homology. Our algorithm also produces a maximal acyclic set of strata, akin to $M$ in Example \ref{graph}. The matroid structure on the set of strata can be harnessed via the greedy algorithm to compute homology without using the Smith normal form.



\begin{rmk}  

The strata and their boundaries give a chain complex which, like the matroid we discuss, is a topological invariant of the CW complex. In dimension two, it appears that this chain complex can be refined to a non abelian version which  gives a complete topological invariant for a class of  two dimensional spaces called \emph{taut complexes} (intuitively, where the attaching maps of strata are locally injective). Taut two complexes exist in every possible homotopy type of connected two complexes. Further details can be found in Remark \ref{rmk:taut}.

\end{rmk}

\medskip\noindent We give preliminary definitions in Section \ref{definitions}. In Section \ref{matroid structure}, we define the notion of a stratum, prove the existence of a chain complex and a matroid structure based on strata, and show that they are topological invariants. Finally in Section \ref{algorithm} we give a cubic-time algorithm, using the invariants of the previous section, that computes a homology basis consisting of minimally supported cycles.

\paragraph*{Acknowledgements}
We would like to thank Daniel Whalen for enlightening conversations regarding time complexity, and Karen Marinez for her indispensable administrative help. The second author acknowledges partial support from ISF grant 1794/14.

\section{Preliminary definitions}
\label{definitions}

The following definitions will be useful in the next sections.

\begin{defn}
By \emph{coordinate space} we mean a finite dimensional real vector space $V$ with a choice of $\text{dim}(V)$ co-dimension one subspaces in general position. By general position we mean that the dimension of the intersection of any $n$ of these hyperplanes is $\text{dim}(V)-n$.
\end{defn}

\begin{defn} \cite{Steenrod67} A \emph{regular cell complex} is a CW complex in which the attaching maps of the cells are homeomorphisms onto their images.
\end{defn}

\begin{defn} \cite{Harary69a}
A \emph{matroid} is a finite set $E$, called the \emph{ground set}, together with a family of subsets $I$ called \emph{independent sets} such that

\begin{itemize}
\item[1)] The empty set is independent.
\item[2)] If $A\subset B$ and $B$ is independent, then $A$ is independent.
\item[3)] If $A$ and $B$ are independent and $B$ has more elements than $A$, then there exists a $b\in B\setminus A$ such that $A \cup \{b\}$ is independent.
\end{itemize}

\noindent Subsets of $E$ that are not independent are called \emph{dependent}.
\end{defn}

\begin{ex}
The edges of a finite graph have the structure of a matroid, where subsets forming trees are independent. Note that the cycles of the graph are exactly the minimal dependent sets of the matroid.
\end{ex}

\begin{defn} \cite{blandlasvergnas} A \emph{signed set} $A$ is a map of sets whose target is $\{ +, -\}$. The \emph{support} of $A$ is its domain. We denote the preimage of $+$ and $-$ by $A^+$ and $A^-$ respectively, and we denote $A$ postcomposed with the involution on $\{-, +\}$ by $-A$. Thus for example, one can write $A^+= (-A)^-$.
\end{defn}

\begin{defn} \cite{blandlasvergnas}
\label{oriented matroid}
An \emph{oriented matroid} consists of a finite ground set $E$ together with a collection of signed sets $\mathscr{C}$. Each signed set in $\mathscr{C}$ is supported on a subset of $E$. The signed sets have the following properties.

\begin{itemize}
\item[1)] The empty set is not in $\mathscr{C}$.
\item[2)] If $A$ is in $\mathscr{C}$ then $-A$ is in $\mathscr{C}$.
\item[3)] If $A,B \in \mathscr{C}$ and $A \subseteq B$ in $\mathscr{C}$, then either $A = B$ or $A = -B$
\item[4)] For all $A,B \in \mathscr{C}$ such that $A \neq -B$ with an element $e$ in $A^+ \cap B^-$, there exists $Z$ in $\mathscr{C}$ such that $Z^+ \subset (A^+ \cup B^+)\setminus\{e\}$ and $A^- \subset (A^- \cup B^-)\setminus\{e\}$

\end{itemize}

\end{defn}

\begin{ex}
The edges of a finite directed graph have the structure of an oriented matroid as follows. A signed subset $A$ of edges is in $\mathscr{C}$ if it forms a cycle in the underlying undirected graph with the following property. The directions of edges in $A^+$  are consistent with an orientation for the cycle, while the directions of edges  in $A^-$ are inconsistent.\\
\end{ex}

\noindent An oriented matroid has an underlying (unoriented) matroid. The support of each signed set in $\mathscr{C}$ is a minimal dependent set of the underlying matroid.

\begin{defn}
\label{defn_pseudomanifold}
\cite{Brasselet-1996}
A topological space $X$ is called a $d$-\emph{pseudo-manifold with boundary} if there exists a triangulation $K$ with the following properties.
\begin{itemize}
\item[1)] $X$ is the union of all the simplices in $K$.
\item[2)] Every $(d-1)$-simplex is in the boundary of exactly one or two $d$-simplices. 
\item[3)] For every pair of $d$-simplices $\sigma$, $\sigma' \in K$ there is a sequence of $d$-simplices  $\sigma = \sigma_1\ldots \sigma_l=\sigma'$ such that $\sigma_i\cap\sigma_{i+1}$ is a $(d-1)$-simplex for all $i$.

\end{itemize}

\noindent A point in a topological space $X$ is called a $d$-\emph{pseudo-manifold point} if it has a neighborhood which is a pseudo-manifold with boundary

\end{defn}

\begin{figure}[h]
\begin{center}
\includegraphics[width=3in]{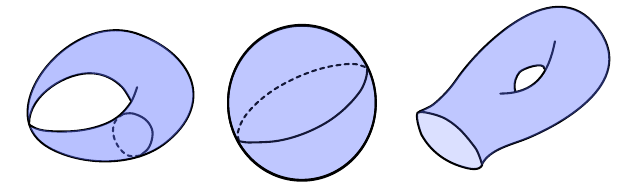}
\caption{Examples of two-dimensional pseudo-manifolds with boundary (individual simplices not shown).}
\end{center}
\end{figure}

\begin{figure}[h]
\begin{center}
\includegraphics[width=3in]{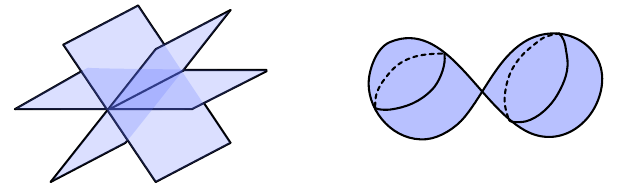}
\caption{Examples of simplicial complexes that are not pseudo-manifolds. The first example violates (2), while the second example violates (3). Individual simplices not shown.}
\end{center}
\end{figure}

\medskip\noindent Note that pseudo-manifolds, as referred to in this paper, are connected by definition.

\section{Strata and geometric matroid structure}
\label{matroid structure}

Let $X$ be a compact Hausdorff space underlying a finite simplicial complex of dimension $d$ (equivalently $X$ underlies a finite regular cell complex structure of dimension $d$, see  \cite{Steenrod67}).

\medskip\noindent Remove the $d$-pseudo-manifold points from $X$ to obtain a closed subset $X_{d-1}$. Then remove the $(d-1)$-pseudo-manifold points from $X_{d-1}$ to obtain $X_{d-2}$ and so on.
This produces a finite decreasing filtration of closed sets. These are subcomplexes (see \cite[Chapter I, Section 4]{Steenrod67}) for any regular cell complex structure on $X$. Note that $X_{d-1}$ is comprised of precisely the $(d-1)$-cells of $X$ that are in the intersection of more than two $d$-cells.

\medskip\noindent For every pair $(X_k, X_{k-1})$ there is a long exact sequence: 

\begin{equation*}
\ldots \xrightarrow{} H_{n+1}(X_k,X_{k-1}) \xrightarrow{\delta_{n+1}} H_{n}(X_{k-1}) \xrightarrow{} H_{n}(X_k) \xrightarrow{\iota_{n}} H_{n}(X_k,X_{k-1}) \xrightarrow{} \ldots
\end{equation*}

\medskip\noindent Using $\delta$ from the exact sequence for $(X_k,X_{k-1})$ and $\iota$ from the exact sequence for $(X_{k-1},X_{k-2})$, we get the following composition:

\begin{equation*}
H_{k}(X_{k},X_{k-1}) \xrightarrow{\delta_k} H_{k-1}(X_{k-1}) \xrightarrow{\iota_{k-1}}  H_{k-1}(X_{k-1},X_{k-1}).
\end{equation*}

\noindent Since $\delta$ is the boundary map, and $\iota$ is inclusion, the composition $\partial_{k} := \iota_{k-1} \circ \delta_k$ is a \emph{boundary} map $H_k (X_k,X_{k-1}) \xrightarrow{} H_{k-1}(X_{k-1},X_{k-2})$. We obtain and the following sequence of groups and morphisms

\begin{equation}
\label{chain}
\ldots \xrightarrow{} H_{k+1}(X_{k+1},X_k) \xrightarrow{\partial_{k+1}} H_k(X_k,X_{k-1}) \xrightarrow{\partial_{k}} H_{k-1}(X_{k-1},X_{k-2}) \xrightarrow{} \ldots
\end{equation}

\begin{defn}
We call the connected components of the set of $k$-pseudo-manifold points of $X_k$ the $k$-\emph{strata} of $X$. In other words, a $k$-\emph{stratum} is a connected component of $X_k \setminus X_{k-1}$.
\end{defn}

\begin{figure}[H]
\label{fig_strata}
\begin{center}
\includegraphics[width=4in]{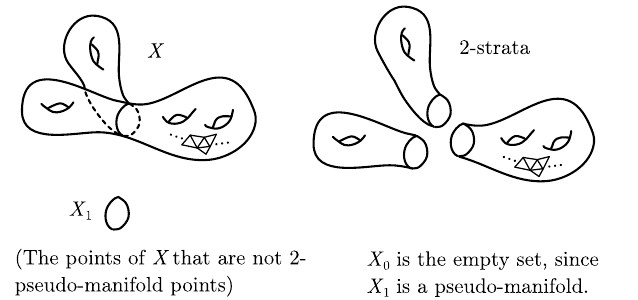}
\caption{A two-dimensional simplicial complex $X$ is shown, followed by its $2$-strata, and $X_1$.}
\end{center}
\end{figure}

\medskip\noindent It may be instructive to see how to a stratum is computed, see Algorithm \ref{alg_stratum} in Section \ref{algorithm}.

\begin{defn} A $k$-stratum is called \emph{orientable} if the $k^{\text{th}}$ homology of its closure in $X_k$, relative its boundary is infinite cyclic (as opposed to $\mathbb{Z}_2$)
\end{defn}

\medskip\noindent One can think of $X$ as made up of strata, with each stratum having an attaching map and a boundary, analogous to a CW complex structure. In the proof of the main theorem below, one sees that the orientable $k$-strata correspond to  natural generators, up to a sign, of  the groups $H_k(X_k, X_{k-1})$, and $\partial_k$ correspond to boundary maps.


\begin{Theorem}

\label{thm:main}

The filtration $X=X_{d} \supseteq X_{d-1} \supseteq \ldots$ has the following properties:
\begin{itemize}
\item[(i)] The maps induced by the exact sequences of pairs make (\ref{chain}) into a chain complex whose top homology is isomorphic to the top homology of $X$.
 \item[(ii)] With real coefficients, each group $C_k := H_k(X_k,X_{k-1})$ has the structure of a coordinate space  whose  codimension one hyperplanes are defined by elements of $C_k$ not supported on a particular $k$-stratum.
 
 \item[(iii)] The set of $k$-strata with chosen orientations has the structure of an oriented matroid.
 \item[(iv)] Any homeomorphism between two such spaces preserving orientations induces an isomorphism between the above structures.
\end{itemize}

\end{Theorem}

\begin{proof}
We prove each part of Theorem \ref{thm:main} in order.
 \begin{itemize}
 \item[i)] We show (by a familiar argument) that (\ref{chain}) is a chain complex. The map $\partial_k$ is the composition of $H_{k+1}(X_{k+1},X_k) \xrightarrow{\delta_{k+1}} H_{k}(X_{k})$ and $H_{k}(X_{k}) \xrightarrow{\iota_k} H_{k}(X_{k}, X_{k-1})$. So $\partial_k \partial_{k-1} = \iota_{k-1} \circ \delta_k \circ \iota_{k-2} \circ \delta_{k-1}$. Since $(\delta)^2 = 0$, and $\iota$ is simply inclusion, we have $\partial^2=0$ and the sequence (\ref{chain}) of groups and morphisms forms a chain complex.\\ \\
   To find the top homology, consider the following part of the chain complex,
    \begin{equation*}
     0 \xrightarrow{} H_d(X_d,X_{d-1}) \xrightarrow{\partial_d} H_{d-1}(X_{d-1},X_{d-2}).
    \end{equation*}
   It is easily seen that the top homology group is $\text{ker} (\partial_d)$. Now, $\partial_d$ is the composition of the following maps which  come from two different exact sequences,
    \begin{equation*}
     H_d(X_d,X_{d-1}) \xrightarrow{\delta_d} H_{d-1}(X_{d-1}) \xrightarrow{\iota_{d-1}} H_{d-1}(X_{d-1},X_{d-2}).
    \end{equation*}
   The map $\iota_{d-1}$ is injective, as its kernel is equal to the image under a map from $H_{d-1}(X_{d-2}) = 0$ (it is, afterall an inclusion). Thus $\text{ker}\partial_d= \text{ker} \delta_d$. This is the image under the map $\iota_d$ from $H_d(X_d) = H_d(X)$ in the following exact sequence,
    \begin{equation*}
     0 \xrightarrow{} H_d(X) \xrightarrow{\iota_d} H_d(X_d,X_{d-1}) \xrightarrow{} \ldots
    \end{equation*}
   This image, in turn, is isomorphic to $H_d(X)$ and hence the top homology of chain complex (\ref{chain}) is the same as the top homology of the topological space $X$.
 
  \item[ii)] $X_k$ has an underlying cell complex structure which gives rise to chain groups $A_*$ ($\mathbb{R}$ coefficients). $H_k(X_k,X_{k-1})$ is the top homology group for the relative chain complex,

  \begin{equation}
    \label{relative}
    0 \xrightarrow{} A_k(X_k)/A_k(X_{k-1}) \xrightarrow{\delta'_k} A_{k-1}(X_k)/A_{k-1}(X_{k-1}) \xrightarrow{} \ldots
  \end{equation}

    So $H_k(X_k,X_{k-1}) = \text{ker}\delta'_k$  and is a free abelian group for all $k$. Now we construct a basis up to  multiplication by real scalars which gives rise to a coordinate space structure on these relative homology groups.

   For each orientable $k$-stratum, we construct a formal sum of oriented $k$-cells in $A_k$ whose interiors are contained in this stratum. The coefficient of each cell is $+1$ or $-1$, chosen so that the boundary of the formal sum only has $(k-1)$-cells which are in $X_{k-1}$. Hence these sums are  in $\text{ker} \delta'_k$. Note that each $k$-cell of $X_k$ appears in the formal sum of exactly one $k$-stratum. This implies that the sums for the various strata form a linearly independent set in the vector space of $k$-cells. It remains to show that every chain in $X_k$ whose boundary is in $X_{k-1}$ is a linear combination of the sums constructed above.

   Let $B$ be a chain in $X_k$ whose boundary is contained in $X_{k-1}$. Let $\sigma$ be a cell appearing with multiplicity $m$ in $B$ and in the formal sum for the $k$-stratum $\mathscr{S}$. If $\tau$ is another cell in $\mathscr{S}$, it is connected to $\sigma$ by a sequence of $k$-cells in $\mathscr{S}$ with consecutive cells sharing a $(k-1)$-face not in $X_{k-1}$ (by (3) in Definition \ref{defn_pseudomanifold}). This gives a path of alternating $k$- and $(k-1)$-cells from $\sigma$ to $\tau$. There may be  several such sequences. For any such sequence, each consecutive $k$-cell must appear in $B$ with multiplicity $m$ and appropriate sign so that the boundary of $B$  is not supported on the $(k-1)$-cells that came before it in the sequence. This is true whether or not $\mathscr{S}$ is orientable.
   
\begin{figure}[h]
\label{fig_pathofcells}
\begin{center}
\includegraphics[width=2in]{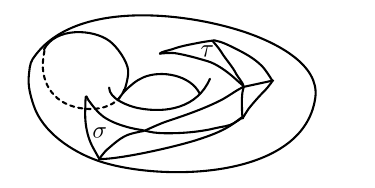}
\caption{An example of a simplicial complex, showing triangles $\sigma$ and $\tau$ and a sequence of alternating triangles and edges connecting them.}
\end{center}
\end{figure}

   Suppose $\mathscr{S}$ is not orientable. Two connecting sequences exist such that the  orientations on a cell $\tau$ induced from a given orientation of some cell $\sigma$ by these two sequences do not agree. Therefore, the multiplicity of $\tau$ in $B$ must be both $+m$ and $-m$. So $m$ must be zero.

   If $\mathscr{S}$ is orientable, the above discussion shows that the entire sum associated with $\mathscr{S}$ appears in $B$ with multiplicity $m$ (respectively with multiplicity zero if $\mathscr{S}$ is not orientable).

   This proves that every element in $C_k$ is a linear combination of oriented strata. A particular co-dimension one hyperplane of the coordinate space structure on $C_k$ consists of those elements which have zero coefficient on a particular stratum.


  \item[iii)] The ground set $E$ of our oriented matroid is the set of oriented $k$ strata. Now, associate to each $k$ stratum a vector in $C_{k-1}$ which corresponds to its boundary. Independence in the matroid is defined as the linear independence of the corresponding boundary vectors in $C_{k-1}$. A signed set $X$ supported on $E$ is in $\mathscr{C}$ if and only if $\{ \partial x \mid x \in X\} \in C_{k-1}$ is minimally linearly dependent, and a dependence relation can be written such that the strata whose boundary vectors appear with positive coefficients form $X^+$ and the strata whose boundary vectors appear with negative coefficients form $X^-$. Note that a minimal linearly dependent set has a dependence relation that is well defined up to scalar multiplication.

We briefly check each of the four conditions in the definition of an oriented matroid.
  
  \begin{enumerate}
\item The empty set $\emptyset \in E$ has boundary $\partial \emptyset$ which is equal to the empty set of vectors in $C_{k-1}$. It is vacuously linearly independent, and so $\emptyset$ is not in $\mathscr{C}$.
\item If $A \in \mathscr{C}$, then a linear combination of the strata in $X$ has zero boundary. If we multiply this linear combination by $-1$ we find that $-A \in \mathscr{C}$.
\item If $A, B \in \mathscr{C}$ and $A \subset B$ we must have $A=B$ or $A=-B$ by minimality of $B$.
\item If $A, B \in \mathscr{C}$, $A\neq -B$, and there is an element $e$ in $A^+ \cup B^-$, it is possible to create a new element $Z$ in $\mathscr{C}$ that does not contain $e$ in its support. This is done by solving for $\partial e$ in dependence relations for $A$ and $B$ and subtracting one from another to form a new dependence relation. The vectors in this new relation are a minimal linearly  independent set and the strata whose boundary vectors are summands in this relation form $Z$. It can be checked that $Z$ has the sign properties given in Definition \ref{oriented matroid}.
\end{enumerate}

   \item[iv)] A homeomorphism from $Y$ to $X$ restricts to a homeomorphism from $Y_k$ to $X_k$ for every $k$, where $Y_k$ is obtained from $Y$ the same way in which $X_k$ is obtained from $X$. An isomorphism is induced between the two chain complexes associated with the pairs $(Y_k,Y_{k-1})$ and $(X_k,X_{k-1})$. This, in turn, induces an isomorphism between $H_k(Y_k,Y_{k-1})$ and $H_k(X_k,X_{k-1})$ for every $k$, which respects the boundary maps. Thus there is an isomorphism between the two chain complexes preserving the coordinate space structure.   \end{itemize}

\end{proof}

\medskip\noindent Recall that the top homology group $H_d(X)$ is  equal to the group of cycles in the top dimension. By the proof of Theorem \ref{thm:main} $(ii)$, these cycles  inject into $H_d(X_d, X_{d-1})$. Theorem \ref{thm:main} $(i)$ shows that every element of $H_d(X_d, X_{d-1})$ corresponds to the linear combination of $d$-strata. Together, these facts give the following corollary, which is illustrated in Figure \ref{fig_strata}.

\begin{Cor}
The top-dimensional cycles can be written as  linear combinations of top-dimensional strata.
\end{Cor}

\begin{figure}[H]
\label{fig_cycleslincomb}
\begin{center}
\includegraphics[width=4in]{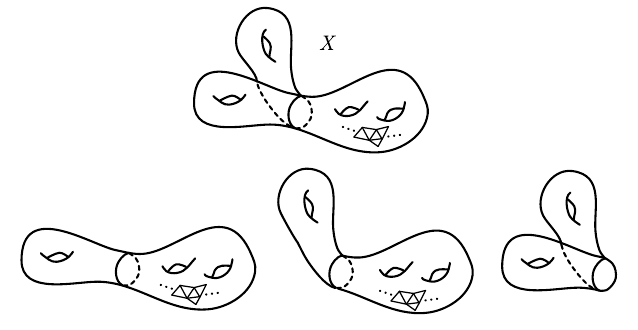}
\caption{A two-dimensional simplicial complex $X$, followed by a basis of its two-dimensional cycles, each element of which is shown as a sum of $2$-strata. Compare with Figure \ref{fig_strata}.}
\end{center}
\end{figure}

\medskip\noindent The chain complex (\ref{chain}) of coordinate spaces is a topological invariant  of the underlying space which, with some enrichment, becomes a complete topological invariant in dimension one and dimension two.

\begin{rmk}
\label{rmk_completeinvariant1}
In dimension one, the connected components of  $1$-pseudo-manifold points are either circles or open intervals. Thus for connected non-trivial spaces, the one-dimensional strata are either the open edges of a graph or a single circle. The circle case is specified by having no degree zero term in the chain complex (\ref{chain}). The graph case for connected spaces of dimension one is  essentially determined by the above coordinate space chain complex. For such a graph, $X_1$ is a finite collection of open intervals and $X_0$ is a finite collection of points. Every open interval and every  vertex  corresponds to a generator of the respective chain groups. The boundary map takes the generator corresponding to an interval to the difference of the generators corresponding to its ending vertices. Thus for every such open interval, which does not form a loop after being attached, the above chain complex completely determines its position in the graph. Therefore in either case (circle or graph) the above chain complex, along with a function giving the number of loops attached to every vertex, is a complete topological invariant for connected compact one-dimensional triangulable spaces.

\begin{figure}[H]
\label{fig_completeinvariant1}
\begin{center}
\includegraphics[width=4in]{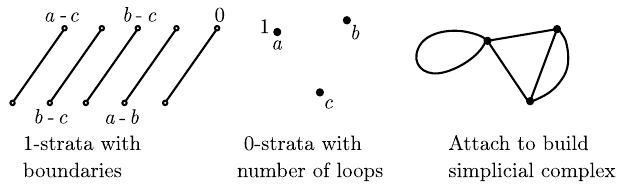}
\caption{Illustration of how to reconstruct $X$ from chain complex (\ref{chain}).}
\end{center}
\end{figure}
\end{rmk}

\begin{rmk}
\label{rmk:taut}
The two-dimensional case proceeds similarly and is a generalization of the one-dimensional case.  
The $1$-skeleton is determined as described in Remark \ref{rmk_completeinvariant1}. The number of two-dimensional strata are specified by the chain complex (\ref{chain}), as in dimension one. We reconstruct the space by attaching the $2$-strata to the $1$-skeleton, much like how one attaches $2$-cells to a $1$-skeleton. The support of the attaching maps is determined by the boundary map. We restrict attention to spaces with locally injective strata-attaching maps (we note that all $2$-complexes are homotopy equivalent to such a complex). In this case, the only additional information needed to determine the attaching maps is the cyclic order in which each boundary component visits the support of the attaching map. Therefore, enriching the chain complex with the topological type of each $2$-stratum, as well as a cyclic word for each of its boundary components (whose letters are the $1$-strata in the support of the boundary) gives a complete homeomorphism invariant of $2$-complex. We leave to the reader the challenge to further develop this picture.

\begin{figure}[H]
\label{fig_completeinvariant2}
\begin{center}
\includegraphics[width=4in]{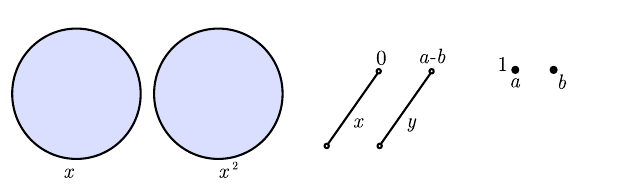}
\caption{The information given by chain complex (\ref{chain}) is shown visually in this example. Each stratum is shown with its boundary information. The enrichment needed to recover the simplicial complex is the the topological type of the $2$-strata (disks) and the $1$ decorating $a$. This denotes that one loop is attached at this $0$-stratum. Assembling these strata according to their boundary information yields a projective plane with a disk attached along a non-trivial loop, with an edge protruding from a point somewhere along this loop.}
\end{center}
\end{figure}
\end{rmk}

\medskip \noindent This concludes the abstract results discussed in this paper. In the following section, we show how these results may be applied algorithmically.

\section{Using the matroid structure to find a minimal support basis for top homology}
\label{algorithm}

We use the matroid structure from Theorem \ref{thm:main} to describe an algorithm that, given a simplicial complex, computes a basis for the top-dimensional homology. Additionally, every element of the computed basis has minimal support in the following sense. No proper subset of a basis element supports a top-dimensional cycle. The algorithm will use the matroid structure in the top two dimensions. Since the ground set of the matroid is the set of orientable strata, we first need to construct the set of strata from the simplices. Our algorithm uses boundary computations. For simplicial complexes, every cell is completely determined by the set of its vertices, and its boundary is determined by subsets that each exclude a vertex. We note that for regular cell complexes, which are determined by the poset of cell inclusions with dimensions attached, the boundary of each cell is part of  this structure. Thus the algorithm can be modified to take regular cell complexes as input.

\medskip\noindent More precisely, let $X$ be a $d$-dimensional simplicial complex, given as a set of vertices $V$ and a set of subsets $\mathscr{V}$ (that each span a simplex of appropriate dimension). The elements of $V$ are ordered, inducing orientations (orderings up to even permutations) on the elements of $\mathscr{V}$. If $\sigma \subset \mathscr{V}$ is a $k$-simplex, $|\sigma|=k+1$. We denote by $-\sigma$ the simplex $\sigma$ taken with opposite orientation. Its boundary $\partial(\sigma)$ is defined as the set of all the subsets of $\sigma$ of size $k$, each with orientation induced by the orientation of $\sigma$ (see \cite{Steenrod67} for details). For a $d$-dimensional simplex $\sigma \subset \mathscr{V}$, we find the stratum $e$ that contains $\sigma$ and determine whether it is orientable as follows.

\begin{algorithm}[H]
\caption{}
\begin{algorithmic}
\label{alg_stratum}
\STATE $e\leftarrow \{\sigma\}$

\STATE $L\leftarrow  \varnothing$

\STATE $P \leftarrow \partial(\sigma)$

\STATE Create a hash table which associates each $(d-1)$-cell to all the $d$-cells that have it in its boundary.

\WHILE{$P$ is not $\varnothing$}

\STATE Pick $\tau$ in $P$.

\IF{there are exactly two cells, $\sigma_1$ (already in $e$) and $\sigma_2$, associated to $\tau$ in the hash table}

\STATE $e \leftarrow e \cup\{-\sigma_2\}$
\STATE $P \leftarrow P \cup \partial(-\sigma_2) $

\ELSE
\STATE $L \leftarrow L \cup \{\tau\}$

\ENDIF

\STATE $P \leftarrow P \backslash (\{\tau\}\cup L)$

\ENDWHILE

\STATE Pick $\sigma$ in $e$

\IF {$-\sigma$ is in $e$}
\STATE $o_e \leftarrow $ non-orientable

\ELSE
\STATE $o_e \leftarrow $ orientable

\ENDIF

\end{algorithmic}
\end{algorithm}




\medskip\noindent The algorithm moves locally between adjacent $d$-dimensional cells, determining whether they fall in the same stratum or not. Two adjacent $d$-cells are in the same stratum when they are the only cells whose boundary contains some $(d-1)$-cell. This computation is seen in the first if loop. As the algorithm adds cells to $e$, the set $P$ is the boundary of $e$ across which we attempt to extend the stratum. In this sense, $P$ is the set of $(d-1)$-cells that need to be processed. Eventually either the stratum is extended across each element of $P$, or the element is sent to $L$, which consists of the boundary of $e$ at the end of the algorithm.

\medskip \noindent If a stratum is non-orientable, it is impossible to assign signs to the cells, in a consistent way. In Algorithm \ref{alg_stratum} this is realized by the fact that if $\sigma$ is in $e$ then $-\sigma$ is forced to be in $e$ as well. We use this fact to determine the orientability of the stratum.

\medskip \noindent We note that Algorithm \ref{alg_stratum} terminates. At the end of each iteration of the while loop, no element that has previously been in $P$ will have been added back. Since we are looking at a finite complex, the while loop will eventually terminate. At the end of the above algorithm $e$ will be the cells forming a strata and $L$ will be the cells forming the boundary of $e$.

\medskip\noindent The hash table is computed by taking each $d$-cell and associating it to the $(d-1)$-cells in its boundary. This process is linear in the number of $d$-cells. With this table in place, each cell in $P$ takes a fixed amount of time to process. This part of the algorithm is linear in the number cells that get processed in $P$. Thus the process of finding all the $d$-strata and their boundaries is linear in the total number of cells.


\medskip\noindent To compute a basis for the top homology, we also need the $(d-1)$-strata, which we construct similarly. Recall that the $(d-1)$-strata are supported on the boundaries of the orientable $d$-strata. Thus, when constructing the $(d-1)$-strata, it suffices to consider only those $d-1$-cells that are in the boundaries of the orientable $d$-strata. We repeat the above process with these $(d-1)$-cells to compute the corresponding $(d-1)$-strata. 


\medskip\noindent We now outline an algorithm for computing a basis of minimal support $d$-cycles (of a $d$-dimensional simplicial complex $X$), given its $d$-strata and $(d-1)$-strata. We use the matroid with the ground set $E= \{\text{orientable } d \text{-strata}\}$. In what follows, all coefficients are rational numbers. Once we have a cycle with rational coefficients, we can multiply it by an appropriate integer factor to get relatively prime integer coefficients. It is helpful to recall that the boundary of a $d$-stratum $\sigma$ is a linear combination of $(d-1)$-strata. By slight abuse of notation, we denote the vector in $C_{d-1}$ given by this linear combination $\partial(\sigma)$. Further, we suggest the reader bear in mind Example \ref{graph} while reading Algorithm \ref{find_cycles}. Note that when Algorithm \ref{find_cycles} is applied to a graph, the procedure from Example \ref{graph} is obtained.

\begin{algorithm}[H]
\caption{}
\begin{algorithmic}
\label{find_cycles}
\STATE $M \leftarrow \varnothing$
\STATE $S \leftarrow \varnothing$
\STATE $E \leftarrow \{\text{orientable } d \text{-strata}\}$

\FOR {all $e \in E$}
\STATE Compute matrix $R$ whose columns are $\{\partial(e)\}\cup \{\partial(m) \text{ such that } m \in M\}$ with respect to the basis for $C_{d-1}$ given by $(d-1)$-strata (see Theorem \ref{thm:main}, (ii)). 

\STATE $D\leftarrow$ The minimal dependence relation in $R$ over the integers obtained by computing its row echelon form. 

\IF {$D$ is $0$}
\STATE $M\leftarrow M\cup \{e\}$
\ELSE 

\STATE $S\leftarrow S\cup \{D\}$

\ENDIF
\ENDFOR

\end{algorithmic}
\end{algorithm}





\medskip \noindent The algorithm adds strata to $M$ as long as $M$ remains an independent set of the matroid $E$. Each time a stratum $e$ cannot be added to $M$, the algorithm finds the cycle created in $M \cup \{e\}$ and adds it to $S$. At the end of the algorithm, the set $M$ is a maximal independent subset of $E$ and $S$ is a basis of cycles for the top-dimensional homology group. In fact, due to the independence of $M$ at all times during Algorithm \ref{find_cycles}, these cycles are minimally supported on the strata and hence this is a minimally supported basis.

\medskip \noindent During each iteration of the for loop, the row echelon form can be computed by adding a new column corresponding to $\partial(e)$, to a previously row-echelonized matrix. This process is quadratic in the number of $(d-1)$-strata and thus the whole algorithm is cubic in the number of $(d-1)$-cells. Therefore the computation of a homology basis (Algorithm \ref{alg_stratum} followed by Algortihm \ref{find_cycles}) is cubic in the number of cells in $X$.


\medskip\noindent We note that the structure of an oriented matroid provides some efficiency for the simplex algorithm \cite{BlandSimplex}. Since the matroid structure defined here is indeed an oriented, the simplex algorithm can likely be used for further applications.

\medskip\noindent Finally, we note that this greedy algorithm solves a higher-dimensional analogue of the minimum spanning tree problem. If the for loop in Algorithm \ref{find_cycles} goes through the elements of $E$ in order of increasing weight, then at the end of the computation, $M$ will be a minimal weight, maximal acyclic set of strata.





\nocite{*}
\bibliography{resubmission}
\bibliographystyle{plain}

\end{document}